\documentclass[12pt,twoside]{article}
\usepackage{amsmath,amsthm}
\usepackage{amssymb}
\usepackage{graphicx}

\date{}

\providecommand{\norm}[1]{\lVert#1\rVert}
\begin{document}
\centerline {\normalsize{\bf SOLUTION OF BURGER EQUATION WITH VISCOSITY}} 
\centerline {\normalsize{\bf APPLYING THE BOUNDARY LAYER THEORY}}


\centerline{}

\centerline{\bf {Oscar Mart\'inez N\'u\~{n}ez}}

\centerline{}

\centerline{Faculty of Exact and Natural Sciences, University of
Cartagena,}

\centerline{Campus San Pablo, Avenue of Consulado,}

\centerline{Cartagena de Indias, Bolivar, Colombia}

\centerline{omartinezn@unicartagena.edu.co}

\centerline{}

\centerline{\bf {Ana Magnolia Mar\'in Ram\'irez}}

\centerline{}

\centerline{Faculty of Exact and Natural Sciences, University of
Cartagena,}

\centerline{Campus San Pablo, Avenue of Consulado,}

\centerline{Cartagena de Indias, Bolivar, Colombia}

\centerline{amarinr@unicartagena.edu.co}

\centerline{}

\centerline{\bf {Rub\'en Dar\'io Ort\'iz Ort\'iz }}

\centerline{}

\centerline{Faculty of Exact and Natural Sciences, University of
Cartagena,}

\centerline{Campus San Pablo, Avenue of Consulado,}

\centerline{Cartagena de Indias, Bolivar, Colombia}

\centerline{rortizo@unicartagena.edu.co}

\newtheorem{Theorem}{\quad Theorem}[section]

\newtheorem{Definition}[Theorem]{\quad Definition}

\newtheorem{Corollary}[Theorem]{\quad Corollary}

\newtheorem{Lemma}[Theorem]{\quad Lemma}

\newtheorem{Example}[Theorem]{\quad Example}

\centerline{}

\begin{abstract}
In this article we find the solution of the Burger equation with viscosity applying the boundary layer theory. In addition, we will observe that the solution of Burger's equation with viscosity converges to the solution of Burger's stationary equation in the norm of $L_{2}([-1,1])$.
\end{abstract}


   
\section{Introduction}
Most mathematical models derived from real problems in physics, as well as from other sciences, are non-linear, a property that involves difficulty in analyzing and determining solutions. However, many of these problems have by nature the property of integrability, which makes their study attractive.\\
The Burgers equation with viscosity was first introduced by Bateman \cite{B} in 1915. This (see, \cite{GL},\cite{GC},\cite{GA},\cite{GG}) is a balance between time evolution, nonlinearity and diffusion. Burgers (1948) first developed this equation primarily to throw light on turbulence described by the interaction of two opposite effects of convection and diffusion. However,
turbulence is more complex in the sense that it is both three-dimensional and
statistically random in nature. The equation is used as a model in many fields such as acoustics \cite{S}, continuous stochastic processes \cite{W}, heat conduction \cite{bluman}. Burgers’ equation can also be considered as a simplified form of the Navier-Stokes equation \cite{bell} due to the form of the nonlinear convection term and the occurrence of a viscosity term.\\
Burgers’ equation is one of the very few nonlinear partial differential equations that can be solved exactly. At present, there are several texts and articles in which the use and calculation of its solution by different methods are shown, for example in \cite{DL} and \cite{GG} the exact solution of said equation is shown, using the transformation Cole-Hopf allows to transform it into a linear diffusion equation through a nonlinear transformation.\\
In this paper, we use a different literature (boundary layer) in order to construct an approximate solution of the burger equation with valid viscosity throughout its domain.

\section{Burger equation with viscosity}
Burger equation with viscosity takes the form
\begin{equation}
\label{Burger}
u_{t}(x,t)+u(x,t)u_{x}(x,t) = \epsilon u_{xx}(x,t), \ \ -1 \leq x \leq 1, \ \epsilon > 0,\ \ t>0
\end{equation}
\begin{equation}
\label{Condicion}
u(-1,t)= \alpha \ \ \ u(1,t)= \beta
\end{equation}
If $u_{t}(x,t)= 0$ for fixed $t$, the equation (\ref{Burger}) becomes
\begin{equation}
\label{Burger estacinaria}
\epsilon u_{xx}(x,t) - u(x,t)u_{x}(x,t) = 0
\end{equation}
known as the stationary burger equation.
\begin{Theorem}
La solution of $(\ref{Burger})$ subject to $(\ref{Condicion})$ is given by:
$$U(x,\epsilon)= -\alpha\tanh\left[\frac{\alpha}{2}\left( \frac{x}{\epsilon}+k \right)  \right]$$
Where \ $U(x,\epsilon) \in L_{2}([-1,1])$. 
\end{Theorem}
\begin{proof}
Suppose $u_{t}(x,t)=0$ for fixed $t$, transforms to (\ref{Burger}) into
\begin{equation}
\label{au}
 \displaystyle \left\{ {\epsilon u^{+}_{xx}(x,t)- u^{+}(x,t) u^{+}_{x}(x,t)=0 \atop  u^{+}(1,t)=\beta } \right. 
\end{equation}
\begin{equation}
\label{ae}
 \displaystyle \left\{ {\epsilon u^{-}_{xx}(x,t)- u^{-}(x,t) u^{-}_{x}(x,t)=0 \atop  u^{-}(-1,t)=\alpha } \right. 
\end{equation}
Let's solve (\ref{au}).\\\\
Let's find, $u^{+,ext}_{0}$, taking
$$u^{+,ext}\sim \sum \epsilon^{n}u^{+,ext}_{n}(x),\ u^{+,ext}_{x}\sim \sum\epsilon^{n}u^{+,ext}_{nx}(x), \ u^{+,ext}_{xx}\sim  \sum\epsilon^{n}u^{+,ext}_{nxx}(x)  $$  
Replacing these expressions in (\ref{au}), we have:
$$\epsilon (\sum\epsilon^{n}u^{+,ext}_{nxx}(x)) - (\sum \epsilon^{n}u^{+,ext}_{n}(x))(\sum\epsilon^{n}u^{+,ext}_{nx}(x))= 0$$
From here,
$$ O(\epsilon^{0}): - u^{+,ext}_{0}u^{+,ext}_{0x}=0$$
$$ O(\epsilon): u^{+,ext}_{0xx}-(u^{+,ext}_{0}u^{+,ext}_{1x}+u^{+,ext}_{1} u^{+,ext}_{0x})=0 $$\\
Suppose $u_{0}^{+,ext}\neq 0,$ 
so $u_{0x}^{+,ext}=0,$ then $u_{0}^{+,ext}=k,$ thus $u_{0}^{+,ext}=\beta.$\\\\
Let's find $u^{+,int}_{0}$\\\\
Let $u^{+,int}(s)$ an expansion near the source where $s=\frac{x}{\epsilon^{\gamma}}$, replacing in (\ref{au}) we have
\begin{equation}
\label{ad}
\centerline{$\epsilon^{1 - 2\gamma}u_{ss}^{+,int}-\epsilon^{-\gamma}u^{+,int}u_{s}^{+,int}=0$}
\end{equation}
Balancing (\ref{ad}) $\gamma = 1$.\\\\
Taking\\
$$u^{+,int} \sim \sum\epsilon^{n}u^{+,int}_{n}(x), \ \  u^{+,int}_{s} \sim\sum\epsilon^{n}u^{+,int}_{ns}(x)\ \ \mbox{y} \ \ u^{+,int}_{ss}\sim \sum\epsilon^{n}u^{+,int}_{nss}(x)$$
Replacing these expressions in (\ref{ad}), we have:

$$\epsilon^{-1}(\sum\epsilon^{n}u^{+,int}_{nss}(x))
-\epsilon^{-1}(\sum \epsilon^{n}u^{+,int}_{n}(x))(\sum\epsilon^{n}u^{+,int}_{ns}(x))=0 $$\
From here,
\begin{equation}
\label{af}
\centerline{$O(\epsilon^{-1}):u^{+,int}_{0ss} -u^{+,int}_{0} u^{+,int}_{0s}=0$}
\end{equation}
Let's take the substitution $w=u_{0s}^{+,int},$ $w_{s}=w_{ u_{0}^{+,int}}w,$ replacing in (\ref{af}) we have $w\left( w_{u_{0}^{+,int}}- u_{0}^{+,int} \right)=0$ if and only if $w=0$ \ o \  $w_{ u_{0}^{+,int}}- u_{0}^{+,int}=0$.\\
Let's suppose $w \neq 0,$ so $w_{u_{0}^{+,int}}- u_{0}^{+,int} =0,$ 
using variable separation and integrating:
$$w=\frac{\left( u_{0}^{+,int} \right)^{2}}{2}+c_{0}.$$
Then
$$u_{0s}^{+,int}=\frac{\left( u_{0}^{+,int} \right)^{2}}{2}+c_{0} $$
Applying variable separation again, integrating and 
replacing $c_{0}$ por $-c_{0},$ we have
$$\frac{-\sqrt{2c_{0}}}{c_{0}}\tanh^{-1}\left( \frac{u_{0}^{+,int}}{\sqrt{2c_{0}}} \right)+c_{1}=s+c_{2} $$
Thus,
$$u_{0}^{+,int}=-\sqrt{2c_{0}}\tanh\left[ \frac{\sqrt{2c_{0}}}{2}\left(s +c_{3} \right) \right]$$
Using prandtl's glue condition
$$\lim_{s\rightarrow +\infty}u_{0}^{+,int}=\lim_{x\rightarrow 0^{+}}u_{0}^{+,ext}$$
we have
$$\lim_{s\rightarrow +\infty}\left\{-\sqrt{2c_{0}}\tanh\left[ \frac{\sqrt{2c_{0}}}{2}\left(s +c_{3} \right) \right] \right\} = \lim_{x\rightarrow 0^{+}} \beta$$
Given the $\tanh\left[ \frac{\sqrt{2c_{0}}}{2}\left(s +c_{3} \right) \right] \rightarrow 1,$ 
when $s\rightarrow +\infty$, so
$$u_{0}^{+,int}=-\beta\tanh\left[  \frac{\beta}{2}\left( \frac{x}{\epsilon} +k\right) \right] $$
Doing an analogous process to $u^{-},$ we have
$$u_{0}^{-,ext}=\alpha \ \ \mbox{y} \ \ u_{0}^{-,int}= -\sqrt{2c_{0}}\tanh\left[ \frac{\sqrt{2c_{0}}}{2}\left(s +c_{3} \right) \right]$$
As $\tanh\left[ \frac{\sqrt{2c_{0}}}{2}\left(s +c_{3} \right) \right] \rightarrow -1,$ when $s\rightarrow -\infty,$ 
it follows that $\alpha = -\beta.$\\\\
Thus
$$
\left\{ \begin{array}{rcl}
u_{0}^{-,ext} &=& \alpha \\
u_{0}^{int}&=& -\alpha\tanh\left[  \frac{\alpha}{2}\left( \frac{x}{\epsilon} +k\right) \right]  \\
 u_{0}^{+,ext}&=&-\alpha
\end{array}\right.$$\\

Let's find $U^{comp}$
\begin{itemize}
\item[i.]$$
\begin{array}{rcl}
U^{comp}&=&u_{0}^{int}+u_{0}^{-,ext}-u_{0}^{mach} \\ \\
 &=& -\alpha\tanh\left[  \frac{\alpha}{2}\left( \frac{x}{\epsilon} +k\right) \right] + \alpha -\alpha \\ \\
 &=&-\alpha\tanh\left[  \frac{\alpha}{2}\left( \frac{x}{\epsilon} +k\right) \right]
\end{array}$$
\item[ii.]$$
\begin{array}{rcl}
U^{comp}&=&u_{0}^{int}+u_{0}^{+,ext}-u_{0}^{mach} \\ \\
 &=& -\alpha\tanh\left[  \frac{\alpha}{2}\left( \frac{x}{\epsilon} +k\right) \right] + (-\alpha) -(-\alpha) \\ \\
 &=&-\alpha\tanh\left[  \frac{\alpha}{2}\left( \frac{x}{\epsilon} +k\right) \right]
\end{array}$$
\end{itemize}
Thus 
$$U(x,\epsilon)=-\alpha\tanh\left[  \frac{\alpha}{2}\left( \frac{x}{\epsilon} +k\right) \right]$$\\

Show that $U(x,\epsilon) \in L_{2}([-1,1]).$\\[0.3cm]
Indeed:

$$
\begin{array}{rcl}
\vert U(x,\epsilon)\vert^{2}&=& \vert -\alpha \tanh\left[  \frac{\alpha}{2}\left( \frac{x}{\epsilon} +k\right) \right] \vert^{2} \\ \\
 &=& \vert \alpha\vert^{2} \vert\tanh\left[  \frac{\alpha}{2}\left( \frac{x}{\epsilon} +k\right) \right]\vert^{2}  \\ \\
 & < & \left\vert\alpha\right\vert^{2}
\end{array}$$
So,
$$\int_{-1}^{1}\vert U(x,\epsilon)\vert^{2}dx < \int_{-1}^{1}\vert \alpha\vert^{2}dx = 2\vert \alpha\vert^{2} < \infty.$$

\end{proof}

\section{Metastability of U(x,$\epsilon$)}
The unique equilibrium solution U(x,$\epsilon$) metastable, for the problem (\ref{Burger}), is given asymptotically by
$$U(x,\epsilon)=-\alpha\tanh\left[  \frac{\alpha}{2}\left( \frac{x}{\epsilon} +k\right) \right]$$
Let us determine the stability of this solution by linealizing the Burger equation with viscosity, having regard to the expression
$$u(x,t) = U(x,\epsilon) + \nu e^{-\lambda t}\Phi(x), \ \ \nu \ll 1.$$
Then 
\begin{align*}
-\lambda \nu e^{-\lambda t}\Phi(x)&= \epsilon \left(U_{xx}(x,\epsilon) + \nu e^{-\lambda t}\Phi_{xx}(x) \right) - \left(U(x,\epsilon) + \nu e^{-\lambda t}\Phi(x)\right)\left( U_{x}(x,\epsilon) + \nu e^{-\lambda t}\Phi_{x}(x)\right)\\ 
&= \left( \epsilon \nu e^{-\lambda t}\Phi_{xx}(x) - U(x, \epsilon)\nu e^{-\lambda t}\Phi_{x}(x) - \nu e^{-\lambda t}\Phi(x)u_{x}(x, \epsilon) \right) + \Gamma (x,t).
\end{align*}
with \ $\Gamma (x,t) = \epsilon U_{xx}(x,\epsilon)- U(x,\epsilon)u_{x}(x,\epsilon) - \nu^{2}e^{-2\lambda t}\Phi(x)\Phi_{x}(x)$.\\\\
Where,
$$O(\nu): \  -\lambda e^{-\lambda t}\Phi(x) = \epsilon e^{-\lambda t}\Phi_{xx}(x) - U(x, \epsilon)e^{-\lambda t}\Phi_{x}(x) - U_{x}(x, \epsilon)e^{-\lambda t}\Phi(x)$$

$$-\lambda \Phi(x) = \epsilon \Phi_{xx}(x) - U(x, \epsilon)\Phi_{x}(x) -\Phi(x)U_{x}(x, \epsilon)$$
So, $$\epsilon \Phi_{xx}(x) - \left(U(x,\epsilon)\Phi(x)\right)_{x} + \lambda \Phi (x) = 0$$
Taking,
\begin{enumerate}
\item[i.]$x = -1, \ \alpha = 1 \ \  \mbox{y} \ \   k = 0, \  \  U(-1,\epsilon) = \tanh\left(\dfrac{1}{2 \epsilon}\right) \longrightarrow 1$, \ when $\epsilon \longrightarrow 0.$
\item[ii.]$x = 1, \ \alpha = 1 \ \  \mbox{y} \ \   k = 0, \  \  U(1,\epsilon) = -\tanh\left(\dfrac{1}{2 \epsilon}\right) \longrightarrow -1$, \ when $\epsilon \longrightarrow 0.$
\end{enumerate}
Now,
\begin{enumerate}
\item[i.]$1 = u(-1,t) = U(-1,\epsilon) + \nu e^{-\lambda t}\Phi(-1) = 1 + \nu e^{-\lambda t}\Phi(-1)$, $\Phi(-1) = 0.$
\item[ii.]$-1 = u(1,t) = U(1,\epsilon) + \nu e^{-\lambda t}\Phi(1) = -1 + \nu e^{-\lambda t}\Phi(1)$, $\Phi(1) = 0.$
\end{enumerate}
Consequently $\Phi$ satisfies the eigenvalue problem 
$$\epsilon \Phi_{xx}(x) - \left(U(x,\epsilon)\Phi(x)\right)_{x} + \lambda \Phi (x) = 0, \ \  -1 \leq x \leq 1, \ \ \Phi(\pm 1)= 0.$$
Defining $\phi(x) \equiv \exp \left[ -\epsilon^{-1}\displaystyle \int_{0}^{x}u(s,\epsilon)ds\right]\Phi(x)$ \ con \  $\Phi(\pm 1)= 0$, we observe that $\phi$ satisfies the own value problem
\begin{equation}
\label{est0}
\epsilon \phi_{xx}(x) + U(x,\epsilon)\phi_{x}(x) + \lambda \phi (x) = 0, \ \  -1 \leq x \leq 1, \ \ \phi(\pm 1)= 0
\end{equation}
Indeed:
$$\epsilon \Phi_{xx}(x) - \left(U(x,\epsilon)\Phi(x)\right)_{x} + \lambda \Phi (x) = 0.$$
\begin{equation}
\label{est}
\epsilon \Phi_{xx}(x) - U(x,\epsilon)\Phi_{x}(x) - U_{x}(x,\epsilon)\Phi_{x}(x) + \lambda \Phi (x) = 0
\end{equation}
Muttiplying (\ref{est}) by 
$$\exp \left[ -\epsilon^{-1}\displaystyle \int_{0}^{x}u(s,\epsilon)ds\right].$$
We have, 
$$ \epsilon \exp \left[ -\epsilon^{-1}\displaystyle \int_{0}^{x}u(s,\epsilon)ds\right] \Phi_{xx}(x) -  U(x,\epsilon)\exp \left[ -\epsilon^{-1}\displaystyle \int_{0}^{x}u(s,\epsilon)ds\right] \Phi_{x}(x) \ - $$
$$- \  U_{x}(x,\epsilon)\exp \left[ -\epsilon^{-1}\displaystyle \int_{0}^{x}u(s,\epsilon)ds\right]\Phi_{x}(x) + \lambda \exp \left[ -\epsilon^{-1}\displaystyle \int_{0}^{x}u(s,\epsilon)ds\right]\Phi (x)  = 0$$
Then,
$$\epsilon \exp \left[ -\epsilon^{-1}\displaystyle \int_{0}^{x}u(s,\epsilon)ds\right] \Phi_{xx}(x) - 2U(x,\epsilon) \exp \left[ -\epsilon^{-1}\displaystyle \int_{0}^{x}u(s,\epsilon)ds\right] \Phi_{x}(x) + $$
$$+ \left( \epsilon^{-1}U^{2}(x,\epsilon)\exp \left[ -\epsilon^{-1}\displaystyle \int_{0}^{x}u(s,\epsilon)ds\right] - U_{x}(x, \epsilon)\exp \left[ -\epsilon^{-1}\displaystyle \int_{0}^{x}u(s,\epsilon)ds\right]\right)\Phi(x) + $$
$$+\  U(x,\epsilon)\exp \left[ -\epsilon^{-1}\displaystyle \int_{0}^{x}u(s,\epsilon)ds\right]\Phi_{x} - \epsilon^{-1}U^{2}(x,\epsilon)\exp \left[ -\epsilon^{-1}\displaystyle \int_{0}^{x}u(s,\epsilon)ds\right]\Phi(x) + $$
$$ + \ \lambda \exp \left[ -\epsilon^{-1}\displaystyle \int_{0}^{x}u(s,\epsilon)ds\right]\Phi(x) = 0 $$
From here, 
$$\epsilon \phi_{xx}(x) + U \phi_{x}(x) + \lambda \phi(x) = 0$$
Where, 
\begin{enumerate}
\item[i.]$\phi(x) \equiv \exp \left[ -\epsilon^{-1}\displaystyle \int_{0}^{x}u(s,\epsilon)ds\right]\Phi(x)$
\item[ii.]$\phi_{x}(x) \equiv - \epsilon^{-1}U(x, \epsilon)\exp \left[ -\epsilon^{-1}\displaystyle \int_{0}^{x}u(s,\epsilon)ds\right]\Phi(x)\ +\  \exp \left[ -\epsilon^{-1}\displaystyle \int_{0}^{x}u(s,\epsilon)ds\right]\Phi_{x}(x)$
\item[iii.]$\phi_{xx}(x) \equiv \exp \left[ -\epsilon^{-1}\displaystyle \int_{0}^{x}u(s,\epsilon)ds\right]\Phi_{xx}(x) - 2\epsilon^{-1}U(x,\epsilon) \exp \left[ -\epsilon^{-1}\displaystyle \int_{0}^{x}u(s,\epsilon)ds\right] \Phi_{x}(x)$
$+ \left( \epsilon^{2}U^{2}(x,\epsilon) \exp \left[ -\epsilon^{-1}\displaystyle \int_{0}^{x}u(s,\epsilon)ds\right] -\epsilon^{-1}U_{x}(x,\epsilon)\exp \left[ -\epsilon^{-1}\displaystyle \int_{0}^{x}u(s,\epsilon)ds\right]\right)\Phi(x)$
\end{enumerate}
We calculate $\phi(\pm1)$
\begin{enumerate}
\item[i.]$\phi(-1) \equiv \exp \left[ -\epsilon^{-1}\displaystyle \int_{0}^{-1}u(s,\epsilon)ds\right]\Phi(-1) = 0$
\item[ii.]$\phi(1) \equiv \exp \left[ -\epsilon^{-1}\displaystyle \int_{0}^{1}u(s,\epsilon)ds\right]\Phi(1) = 0.$
\end{enumerate} 
Since $U$ is monotonically decreasing by $x$ and zero at the turning point $x = 0$, the nature of the turning point for (\ref{est0}) is closely related to the eigenvalue problem $\epsilon \phi{''} - x\phi^{'} +  \lambda \phi = 0,\ \  -1 < x < 1, \ \ \phi(\pm 1) = 0$, which has an exponentially small eigenvalue. In (see, \cite{GP}), it was shows that (\ref{est0}) has as asympotically exponentially small principal eigenvalue and in (see, \cite{GQ}) 
it was obtained asymtotic formula for this function eigenvalue, in turn in these articles demonstrated the exisence of $\phi$. \\
On the other hand, let us see that $u(x,t) = U(x,\epsilon) + \nu e^{-\lambda t}\Phi(x)$ \  with  $\nu \ll1$ converge a $U(x, \epsilon)$ in the norm de $L_{2}$ when $t \longrightarrow \infty$.\\
Indeed:\\
Applying the derivation rule under the integral (see, \cite{D}), we have
\begin{eqnarray*} 
\mbox{$\dfrac{\partial}{\partial t}\displaystyle \int_{-1}^{1}\left(\nu e^{-\lambda t}\Phi(x) \right)^{2}dx$} & = & \mbox{$\displaystyle \int_{-1}^{1}\dfrac{\partial}{\partial t}\left(\nu e^{-\lambda t}\Phi(x) \right)^{2}dx$} \\[0.2cm]
& = & \mbox{$\displaystyle \int_{-1}^{1}2\left(\nu e^{-\lambda t}\Phi(x) \right)(-\lambda)\nu e^{-\lambda t}\Phi(x) dx$} \\[0.2cm] 
& = & \mbox{$\displaystyle \int_{-1}^{1}-2\lambda \nu^{2} e^{-2\lambda t}e^{-\lambda t}\Phi^{2}(x) dx$} \\[0.2cm]
& = & \mbox{$-2\lambda \nu^{2} e^{-2\lambda t}e^{-\lambda t}\displaystyle \int_{-1}^{1}\Phi^{2}(x) dx$}
\end{eqnarray*}
Si $\lambda \geqslant 0$, \ $-2\lambda \nu^{2} e^{-2\lambda t}e^{-\lambda t}\displaystyle \int_{-1}^{1}\Phi^{2}(x)dx \leqslant 0$, then  $\norm{\nu e^{-\lambda t}\Phi(x)}_{2} \longrightarrow 0$, $\nu \ll 1$.\\\\
From here,  $$\norm{u(x,t) - U(x, \epsilon)}_{2} \ =\ \norm{\nu e^{-\lambda t}\Phi(x)}_{2}$$
Thus, 
$$\norm{u(x,t) - U(x, \epsilon)}_{2} \longrightarrow 0, \ \mbox{on the norm of} \ L_{2} \ \mbox{when}\ \ t \longrightarrow \infty.$$
Consequently
$$u(x,t) \sim U(x, \epsilon).$$\\

Show that $u(x,t) = U(x,\epsilon) + \nu e^{-\lambda t}\Phi(x) \in L_{2}([-1,1])$ \  with  $\nu \ll1$ and $\Phi(x)$ bounded.\\[0.3cm]
Indeed:\\[0.2cm]
To $\lambda \geq 0,\ \  e^{-\lambda t} < 1$ and there is $M \geq 0$ \ such that $\vert \Phi(x)\vert< \dfrac{\sqrt{M}}{2}.$

$$\vert u(x,t)\vert = \vert U(x,\epsilon) + \nu e^{-\lambda t}\Phi(x) \vert \Longleftrightarrow \vert u(x,t)\vert^{2} = \vert U(x,\epsilon) + \nu e^{-\lambda t}\Phi(x) \vert^{2}$$
Then,
$$\vert u(x,t)\vert^{2} = \vert U(x,\epsilon) + \nu e^{-\lambda t}\Phi(x) \vert^{2} \leq 4 \left(\vert U(x,\epsilon)\vert^{2} + \vert \nu e^{-\lambda t}\Phi(x) \vert^{2} \right)< 4\vert \alpha\vert^{2} + M$$
So,
$$\int_{-1}^{1}\vert u(x,t)\vert^{2}dx < \int_{-1}^{1}(4\vert \alpha\vert^{2} + M)dx = 8\vert \alpha\vert^{2} + 2M < \infty.$$

\section{Conclusion}
After knowing and using a literature other than those already existing to solve nonlinear partial differential equations involving a certain degree of complexity, It was possible to find and establish an approximate solution of the Burger equation with viscosity associated with many physical problems such as sound waves in a viscous, dynamical medium of gases, waves in viscous elastic tubes filled with liquid, modeling of the flow fico, among others. In addition, it was observed that the solution $u(x; t)$ asi of defined converges to $U(x,\epsilon)$ in the norm $L_{2}$ when $t \longrightarrow \infty$.

\end{document}